\documentclass[12pt, reqno]{amsart}
\usepackage{amsmath, amsthm, amscd, amsfonts, amssymb, graphicx, color}
\usepackage[bookmarksnumbered, colorlinks, plainpages]{hyperref}
\hypersetup{colorlinks=true,linkcolor=red, anchorcolor=green, citecolor=cyan, urlcolor=red, filecolor=magenta, pdftoolbar=true}

\newtheorem{theorem}{Theorem}[section]
\newtheorem{lemma}[theorem]{Lemma}
\newtheorem{proposition}[theorem]{Proposition}
\newtheorem{corollary}[theorem]{Corollary}
\theoremstyle{definition}
\newtheorem{definition}[theorem]{Definition}

\theoremstyle{remark}
\newtheorem{remark}[theorem]{Remark}
\numberwithin{equation}{section}
\newcommand{\sH}{\mathfrak{H}}

\newcommand{\cF}{\mathcal{F}}
\newcommand{\cC}{\mathcal{C}}

\begin{document}

\setcounter{page}{1}

\title[On $J$-frames related to maximal definite subspaces]{On $J$-frames related to maximal definite subspaces}

\author[A. Kamuda \MakeLowercase{and} S. Ku\.{z}el]{Alan Kamuda$^1$ \MakeLowercase{and} Sergiusz Ku\.{z}el$^1$$^{*}$}

\address{$^{1}$AGH University of Science and Technology, 30-059  Krak\'{o}w, Poland.}
\email{\textcolor[rgb]{0.00,0.00,0.84}{alan.kamuda@gmail.com;
kuzhel@agh.edu.pl}}



\begin{abstract}
A definition of frames in Krein spaces is proposed which generalizes the concept of $J$-frames defined by
 Giribet et al., J. Math. Anal. Appl. {\bf 393}  (2012), 122-137.
 The difference consists in the fact that a $J$-frame is related to maximal definite subspaces $\mathcal{M}_\pm$
which are not assumed to be uniformly definite.  The latter allows  to  extend the set of $J$-frames.
In particular, some  $J$-orthogonal Schauder bases can be interpreted as $J$-frames.
\end{abstract} \maketitle

\section{Introduction}
Usually, frames are defined in a Hilbert space setting:
let $\sH$  be an (infinite-dimensional) Hilbert space with an inner product $(\cdot,\cdot)$. A frame for $\sH$ is
a family of vectors $\cF=\{f_n\}$  that satisfies inequalities
\begin{equation}\label{e1}
A||f||^2\leq \sum_{n\in\mathbb{N}}{|(f,f_n)|^2}\leq B||f||^2, \qquad f\in \sH,
\end{equation}
for constants $0<A\leq B< \infty$, which are called frame bounds.

The frame bounds in \eqref{e1} deal with the choice of inner product $(\cdot,\cdot)$.
Sometimes, an improper choice may lead to inconveniences. To explain this point
we consider (following \cite{Esmeral}) the indefinite inner product
$[f,g]=\int_{\mathbb{R}}f(x)\overline{g(x)}\omega(x)dx,$
where the real-valued function $\omega(x)$  is continuous.
The linear set $\mathcal{L}$ of square integrable functions endowed with the indefinite inner product
$[\cdot,\cdot]$ forms a Krein space with the operator of fundamental symmetry $J$ determined by the multiplication
by the sign function of $\omega$. The associated Hilbert space coincides with $\mathcal{L}$ endowed with
the positive inner product
$$
(f,g)=\int_{\mathbb{R}}f(x)\overline{g(x)}|\omega(x)|dx, \qquad f,g\in\mathcal{L}.
$$
According to the frame theory, one should work with $(\cdot,\cdot)$ instead of
the indefinite inner product $[\cdot,\cdot]$. However, by passing from $\omega$ to
$|\omega|$, one might lose some desired properties of $\omega$ (differentiability, for example).
Therefore, it looks natural to work with frames defined in terms of an indefinite inner product.
The development of such ideology leads to the definition below.

Denote by $(\mathfrak{H}, [\cdot,\cdot])$ a Krein space with indefinite inner product $[\cdot,\cdot]$ and with the
fundamental symmetry $J$. The associated Hilbert space $(\mathfrak{H}, (\cdot,\cdot))$ is endowed with the positive
inner product $(\cdot,\cdot)=[J\cdot,\cdot]$ and the norm $\|\cdot\|=\sqrt{(\cdot,\cdot)}$.
\begin{definition}[\cite{Esmeral}]\label{def1}
Let $(\mathfrak{H}, [\cdot,\cdot])$ be a Krein space.
A family of vectors $\cF=\{f_n\}$ is called a frame if
$$
A||f||^2\leq \sum_{n\in\mathbb{N}}{|[f,f_n]|^2}\leq B||f||^2, \qquad f\in \sH,
$$
for some constants $0<A\leq B< \infty$.
\end{definition}

Since the operator of fundamental symmetry $J$ is unitary and self-adjoint in $\sH$,
Definition \ref{def1} implies that  $\cF$ is a frame in the Hilbert
space $(\mathfrak{H}, (\cdot,\cdot))$ if and only if $\cF$ is a frame with the same frame bounds in the
Krein space $(\mathfrak{H}, [\cdot,\cdot])$ \cite[Theorem 3.3]{Esmeral}. This means that
such definition just rephrases the conventional definition of frames in terms of Krein spaces.
 It looks natural to consider more general setting, i.e.,  to define frames in the Hilbert space $\sH_W$
 with $W$-metric $(W\cdot,\cdot)$ \cite[Definition 2.12]{new3}, which is the completion of $\sH$ with respect to $(|W|\cdot, \cdot)$, 
 where a bounded self-adjoint operator  $W$ with $\ker{W}=\{0\}$ 
is the `light version' of the operator of fundamental symmetry $J$ \cite[Section 1.6]{AZ}.

Another reason for studying frames in Krein spaces deals with signal processing problem.
The existence of various decompositions of a given Krein space $(\mathfrak{H}, [\cdot,\cdot])$ into direct sums of positive and
negative subspaces (see Section 2 for the Krein spaces terminology) allows one to construct effective filters for the signals considering vectors $f\in\sH$
as disturbances if theirs indefinite inner products $[f,f]$ are close to zero; see the
discussion at the beginning of Section 3 in \cite{Giribet2}.  Such kind of ideas gives rise to
the following definition of $J$-frame:

Let $\cF=\{f_n\}$ be a Bessel sequence in the Hilbert space $(\mathfrak{H}, (\cdot,\cdot))$ with
synthesis operator $T :  l_2(\mathbb{N}) \to \sH$. Denote $\mathbb{N}_+=\{n\in\mathbb{N} :  [f_n,f_n]\geq{0}\}$,
$\mathbb{N}_-=\{n\in\mathbb{N} :  [f_n,f_n]<{0}\}$, consider the orthogonal decomposition
\begin{equation}\label{fr28}
l_2(\mathbb{N})=l_2(\mathbb{N}_+)\oplus{l_2(\mathbb{N}_-)}
\end{equation}
and determine the restrictions of $T$ onto $l_2(\mathbb{N}_\pm)$:
$$
T_{\pm}\{c_n\}=\sum_{n\in\mathbb{N}_{\pm}}c_nf_n, \qquad \{c_n\}\in{l_2(\mathbb{N})}.
$$
\begin{definition}[\cite{Giribet2}]\label{def2}
The Bessel sequence $\cF=\{f_n\}$ is called a $J$-frame if the ranges $\mathcal{R}(T_{\pm})$ are,
respectively, maximal uniformly positive and maximal uniformly negative subspaces of the Krein
space  $(\mathfrak{H}, [\cdot,\cdot])$.
\end{definition}

The requirement of being maximal uniformly definite imposed on $\mathcal{R}(T_{\pm})$  is
sufficiently strong and an elementary analysis carried out in \cite{Giribet2}  shows that each $J$-frame  $\cF$ in the Krein space
 $(\mathfrak{H}, [\cdot,\cdot])$ has to be a conventional frame in the Hilbert space
 $(\mathfrak{H}, (\cdot,\cdot))$.  Moreover, $\mathcal{R}(T_{\pm})=\mathcal{M}_{\pm}:=\overline{\mbox{span}\{f_n: \ n\in \mathbb{N}_{\pm}\}}$
 and the families $\cF_\pm=\{f_n\}_{n\in\mathbb{N}_\pm}$
are conventional frames of the Hilbert spaces $(\mathcal{M}_{\pm}, \pm[\cdot,\cdot])$.

Obviously, each $J$-frame is a frame in the sense of Definition \ref{def1}.
The inverse implication is not true. In particular, there are orthonormal bases of
the Hilbert space $(\mathfrak{H}, (\cdot,\cdot))$ which cannot be $J$-frames.
Indeed, let $\mathcal{L}$ be a hypermaximal neutral subspace of the Krein space
$(\mathfrak{H}, [\cdot,\cdot])$, then $\mathfrak{H}=\mathcal{L}\oplus{J\mathcal{L}}$.
If $\{f_n\}$ is an orthonormal basis of $\mathcal{L}$, then ${\mathcal{F}}=\{f_n\}\cup\{Jf_n\}$ is an
orthonormal basis of $(\mathfrak{H}, (\cdot,\cdot))$.
However, ${\mathcal{F}}$ cannot be a $J$-frame because, $\mathcal{M}_{+}=\sH$
and $\mathcal{M}_{-}=\{0\}$.

In the last few years,  many papers devoted to the development of full scale frame theory based on Definition \ref{def1}
\cite{new2, new3, Esmeral, Esmeral2}  as well as on Definition \ref{def2} \cite{Giribet2, Giribet, new1, Hossein} have been published. 
In particular,
$J$-fusion frames were defined and studied in \cite{new2, new1}.
However, in our opinion, the above definitions do not completely fit  the ideology of Krein spaces
and some modification that provides  deeper insights
into the structural subtleties of frames in Krein spaces is still needed.
The matter is that the concept
of Krein spaces is more reach in contrast to Hilbert ones due to the possibility to generate
infinitely many (not necessarily equivalent) definite inner products $(\cdot,\cdot)$
beginning with given indefinite inner product $[\cdot,\cdot]$.
For this reason, it seems natural to define frames  in a Krein space $(\mathfrak{H}, [\cdot,\cdot])$
in terms of  frame inequalities based on the indefinite inner product $[\cdot,\cdot]$
 without any relation to frames in the associated Hilbert space $(\sH, (\cdot,\cdot))$.

Due to \cite[Theorem 3.9]{Giribet2},  Definition \ref{def2} of  $J$-frames can be rewritten as follows:
$\cF=\{f_n\}$ is a $J$-frame if and only if $\cF$ is a
conventional frame in the associated Hilbert space $(\mathfrak{H}, (\cdot,\cdot))$,
the conditions $\mathcal{M}_{\pm}\cap\mathcal{M}_{\pm}^{[\bot]}=\{0\}$ hold (the $J$-orthogonal complements
$\mathcal{M}_{\pm}^{[\bot]}$  of $\mathcal{M}_{\pm}$ are defined in \eqref{fr10a}), and
there exist constants $0<A_\pm\leq{B_\pm}$ such that
$$
\begin{array}{c}
A_{+}[f,f]\leq \sum_{n\in\mathbb{N}_{+}}{|[f,f_n]|^2}\leq B_{+}[f,f],  \quad  \forall{f}\in\mathcal{M}_+ \vspace{3mm} \\
-A_{-}[f,f]\leq \sum_{n\in\mathbb{N}_{-}}{|[f,f_n]|^2}\leq -B_{-}[f,f], \quad  \forall{f}\in\mathcal{M}_-.
\end{array}
$$

The subspaces
\begin{equation}\label{fr101b}
\mathcal{M}_{\pm}=\overline{\mbox{span}\{\mathcal{F}_{\pm}\}},  \qquad  \cF_\pm=\{f_n\}_{n\in\mathbb{N}_\pm}
\end{equation}
in the above inequalities have to be maximal uniformly definite in the Krein space  $(\mathfrak{H}, [\cdot,\cdot])$.
Weakening this condition,  we generalize the concept of $J$-frames.
\begin{definition}\label{def3}
Let $(\mathfrak{H}, [\cdot,\cdot])$ be a Krein space with fundamental symmetry $J$.
A family of vectors $\cF=\{f_n\}$ is called a $J$-frame if the subspaces $\mathcal{M}_{\pm}$ defined by \eqref{fr101b}
are maximal definite in $(\mathfrak{H}, [\cdot,\cdot])$ and
there are constants $0<A\leq{B}$ such that
\begin{equation}\label{fr5b}
\begin{array}{c}
A{|}[f,f]{|}\leq \sum_{n\in\mathbb{N}_{+}}{|[f,f_n]|^2}\leq B|[f,f]|,  \quad \forall{f}\in\mbox{span}\{\cF_+\}\vspace{3mm} \\
A{|}[f,f]{|}\leq \sum_{n\in\mathbb{N}_{-}}{|[f,f_n]|^2}\leq B|[f,f]|,  \quad \forall{f}\in\mbox{span}\{\cF_-\}.
\end{array}
\end{equation}
\end{definition}

The aim of this work is to develop a theory of $J$-frames which is based on  Definition \ref{def3}.
The results essentially depend on coinciding the direct sum $\mathcal{D}=\mathcal{M}_+\dot+\mathcal{M}_-$
 with $\sH$.  If $\mathcal{D}=\sH$,  then the subspaces  $\mathcal{M}_\pm$
have to be uniformly definite in the Krein space $(\sH, [\cdot,\cdot])$, the conditions of Definition \ref{def2} are satisfied
 and our results in Section \ref{3.1} are close to \cite{Giribet2}, see Remark \ref{rem1}.

If $\mathcal{D}\not=\sH$,  then at least one of $\mathcal{M}_{\pm}$ loses the property of being uniformly definite
and the new inner product $(\cdot,\cdot)_1$ defined on $\mathcal{D}$ is not equivalent to the initial one.
 In this case Definition  \ref{def2} cannot be applied and, moreover,  this case cannot be studied within
framework of Hilbert spaces $\sH_W$ with $W$-metric, see Section \ref{2.2}.  We show that
each $J$-frame $\cF$ can be realized as a conventional frame in the new Hilbert space $(\widehat{\sH}, (\cdot,\cdot)_1)$
(Proposition \ref{prop3}) and the reconstruction formula \eqref{fr61} holds for elements of $\widehat{\sH}$.
The relevant formulas are essentially simplified when the subspaces $\mathcal{M}_{\pm}$ are assumed to be $J$-orthogonal
(Propositions \ref{fr41} -- \ref{fr91}).

Definition \ref{def3} allows one to consider some $J$-orthogonal sequences and
$J$-orthogonal Schauder bases as $J$-frames (Proposition \ref{prop5}, Corollary \ref{fr73}).
This gives the possibility to describe the common part $\mathfrak{D}=\sH\cap\widehat{\sH}$
and the subset $\mathcal{D}_{un}\subset\mathfrak{D}$ where the series
$f=\sum_{n\in\mathbb{N}}c_nf_n$ converges unconditionally in terms
of the corresponding coefficients $\{c_n\}$ (Proposition \ref{fr78}, Theorem \ref{fr75}).

The paper is organized as follows. We begin with an elementary presentation of the Krein spaces theory.
The monograph \cite{AZ} is recommended as complementary reading on the subject.
In Section 3, we show that each $J$-frame can be considered as a conventional frame in some Hilbert space.
The corresponding reconstruction formulas are rewritten in terms of indefinite inner products.
Special attention is paid to the case where the corresponding subspaces $\mathcal{M}_{\pm}$ are $J$-orthogonal.
Section 4 deals with special classes of $J$-frames: $A$-tight, exact, and $J$-orthogonal.

In what follows, $\mathfrak{H}$ means a complex Hilbert space with inner product $(\cdot,\cdot)$ linear in the first argument.
Sometimes, it is useful to specify the inner product associated with $\mathfrak{H}$.
In that case the notation $(\mathfrak{H}, (\cdot,\cdot))$ will be used.
All topological notions refer to the Hilbert space norm topology. For instance, a subspace
of  $(\mathfrak{H}, (\cdot, \cdot))$  is a linear manifold in $\mathfrak{H}$  which is closed with respect  to the norm $\|\cdot\|=\sqrt{(\cdot,\cdot)}$.
The symbols $\mathcal{D}(A)$ and $\mathcal{R}(A)$ denote the domain and the range of a linear operator
$A$. The notation ${\mathcal{L}_1}\oplus{\mathcal{L}_2}$  means the orthogonal (with respect to an inner product)
direct sum of two subspaces $\mathcal{L}_i$.

\section{Elements of Krein spaces theory}
\subsection{Notation and terminology.}  Let $\sH$ be a complex linear space with
an indefinite inner product (indefinite metric) $[\cdot, \cdot]$.
The space $(\sH, [\cdot, \cdot])$ is called \emph{a
Krein space} if $\sH$ admits a  decomposition
\begin{equation}\label{fr7}
\sH=\sH_{+}[\dot{+}]\sH_{-},
\end{equation}
which is an orthogonal (with respect to $[\cdot, \cdot]$) direct sum of two Hilbert spaces $(\sH_+, [\cdot,\cdot])$ and
$(\sH_-, -[\cdot, \cdot])$.
The decomposition \eqref{fr7} is called \emph{fundamental}
 and it induces the positive inner
product
\begin{equation}\label{fr9}
(f,g):=[f_+,g_+]-[f_-,g_-], \quad f=f_++f_-, \ g=g_+-g_-, \ f_\pm, g_\pm\in\sH_\pm
\end{equation}
 and the associated Hilbert space $(\sH, (\cdot,\cdot))$. The operator
\begin{equation}\label{fr10}
 Jf=f_+-f_-, \quad f=f_++f_-, \quad f_\pm\in\sH_\pm
\end{equation}
is unitary and self-adjoint in the Hilbert space $(\sH, (\cdot,\cdot))$ and it is called
\emph{the operator of fundamental symmetry}. The fundamental symmetry $J$
 allows one to express $(\cdot,\cdot)$ in terms of indefinite metric: $(\cdot, \cdot)=[J\cdot,\cdot]$.

A closed subspace $\mathcal{L}$ of the Krein space $({\mathfrak H}, [\cdot,\cdot])$  is called
\emph{neutral, negative, positive} if all nonzero elements $f\in\mathcal{L}$ are, respectively,
neutral $[f,f]=0$, negative $[f,f]<0$ or positive $[f,f]>0$. Further,
$\mathcal{L}$ is called \emph{uniformly positive, uniformly negative} if, respectively,
$[f,f]\geq{\alpha}(f,f)$, \  $-[f,f]\geq{\alpha}(f,f)$ for certain  $\alpha>0$ and all $f\in\mathcal{L}.$
A subspace $\mathcal{L}$ of ${\mathfrak H}$ is called \emph{definite} if it is either positive or negative.
 The term \emph{uniformly definite} is defined accordingly.

In each of the above mentioned classes we can define maximal subspaces.
For instance, a closed positive subspace $\mathcal{L}$ is called
\emph{maximal positive} if $\mathcal{L}$ is not a proper subspace of a positive subspace in $\mathfrak{H}$.
A maximal neutral subspace $\mathcal{L}$ is called
\emph{hypermaximal} if the Hilbert space $(\sH, (\cdot,\cdot))$ admits the decomposition
 $\mathfrak{H}=\mathcal{L}\oplus{J\mathcal{L}}$.

Subspaces $\mathcal{L}_1, \mathcal{L}_2$ of ${\mathfrak H}$
are said to be \emph{$J$-orthogonal} if $[f,g]=0$ for all
$f\in\mathcal{L}_1$ and $g\in\mathcal{L}_2$.
The $J$-orthogonal complement of a subspace $\mathcal{L}$ of $\mathfrak{H}$ is defined as
\begin{equation}\label{fr10a}
{\mathcal L}^{[\bot]}=\{g\in{\mathfrak H}\ :\ [f,g]=0,\ \forall{f}\in
{\mathcal L}\}
\end{equation}
and it is a closed linear subspace of the Hilbert space $(\mathfrak{H}, (\cdot,\cdot))$.

Let $A$ be a  densely defined operator acting in a Krein space $(\sH, [\cdot,\cdot])$. Repeating
the standard definition of the adjoint operator with the use of an indefinite inner
product $[\cdot,\cdot]$ we define the adjoint operator
$A^+$ of $A$. In this case
 $[Af,g]=[f,A^+g]$ for all $f\in\mathcal{D}(A)$ and all $g\in\mathcal{D}(A^+)$.
An operator $A$  is called \emph{$J$-self-adjoint} if $A=A^+$. An operator $A$ is called \emph{$J$-positive} if
$[Af,f]>0$ for $f(\not=0)\in\mathcal{D}(A)$.

\subsection{Positive inner products generated by an indefinite inner product.} \label{2.2}
{\bf I.} Let $\mathcal{M}_{\pm}$ be maximal uniformly definite (positive/negative) subspaces of the Krein space $(\sH, [\cdot,\cdot])$.
It is easy to see from \cite[Lemma 2.1]{Grod}  that
\begin{equation}\label{sas12b}
\sH=\mathcal{M}_{+}\dot{+}\mathcal{M}_{-}.
\end{equation}
The direct sum \eqref{sas12b} determines the operator, cf. \eqref{fr10}
\begin{equation}\label{fr10b}
 {J_{\mathcal{M}}}f=f_{\mathcal{M}_{+}}-f_{\mathcal{M}_{-}}, \quad f=f_{\mathcal{M}_{+}}+f_{\mathcal{M}_{-}}, \quad f_{\mathcal{M}_{\pm}}\in{\mathcal M}_\pm
\end{equation}
and the new positive inner product, cf. \eqref{fr9}
\begin{equation}\label{fr9b}
(f,g)_1:=[f_{\mathcal{M}_{+}},g_{\mathcal{M}_{+}}]-[f_{\mathcal{M}_{-}}, g_{\mathcal{M}_{-}}].
\end{equation}
which is equivalent to $(\cdot,\cdot)$.
The operator $J_{\mathcal{M}}$ is bounded in $\sH$ and $J_{\mathcal{M}}^2=I$.
\begin{lemma}\label{lem2}
Let $(\cdot,\cdot)_1$ be defined by \eqref{fr9b}. Then
\begin{equation}\label{fr14}
(\cdot,\cdot)_1=[\cC\cdot, \cdot], \quad \cC=\frac{1}{2}(J_{\mathcal{M}}+J_{\mathcal{M}}^+),
\end{equation}
where $J_{\mathcal{M}}^+$ is the adjoint of $J_{\mathcal{M}}$ with respect to the indefinite inner product $[\cdot, \cdot]$.
\end{lemma}
\begin{proof}
It follows from \eqref{fr10b}, \eqref{fr9b} that, for all $f,g\in\sH$,
\begin{eqnarray*}
(f,g)_1=\frac{1}{4}([(I+J_{\mathcal{M}})f, (I+J_{\mathcal{M}})g]-[(I-J_{\mathcal{M}})f, (I-J_{\mathcal{M}})g]) & & = \\
\frac{1}{2}([J_{\mathcal{M}}f, g]+[f, J_{\mathcal{M}}g])=\frac{1}{2}[(J_{\mathcal{M}}+J_{\mathcal{M}}^+)f,g] & &
\end{eqnarray*}
that completes the proof.
\end{proof}

Relation  \eqref{fr14} implies  that  $\cC$ is a bounded, $J$-self-adjoint,  and $J$-positive operator in the Krein space
$(\sH, [\cdot,\cdot])$.  Moreover,  $0\in\rho(\cC)$.

{\bf II.} Assume that $\mathcal{M}_{\pm}$ are maximal definite but at least one of
these subspaces loses the property of being uniformly definite. Then the direct sum
\begin{equation}\label{sas12}
\mathcal{D}=\mathcal{M}_{+}\dot{+}\mathcal{M}_{-}
\end{equation}
does not coincide with $\sH$ and $\mathcal{D}$ is a dense set in the Hilbert space $(\sH, (\cdot,\cdot))$.

Similarly to the case ${\bf I}$,  the direct sum \eqref{sas12} generates by the formula \eqref{fr9b}
 a new inner product $(\cdot,\cdot)_1$ defined on $\mathcal{D}$. In contrast to the previous case, the inner product
 $(\cdot,\cdot)_1$ \emph{is not equivalent} to the initial product $(\cdot,\cdot)$ and the linear space
$\mathcal{D}$ endowed with $(\cdot,\cdot)_1$ is a pre-Hilbert space.
Let the Hilbert space $(\widehat{\sH}, (\cdot,\cdot)_1)$  be the completion of $\mathcal{D}$ with respect to $(\cdot,\cdot)_1$.
By construction, it can be decomposed:
\begin{equation}\label{fr30}
\widehat{\sH}=\widehat{\mathcal{M}}_{+}\oplus_1\widehat{\mathcal{M}}_{-},
\end{equation}
where the subspaces $\widehat{\mathcal{M}}_{\pm}$ of $\widehat{\sH}$
are the completion of the subspaces ${\mathcal{M}}_{\pm}$ and $\oplus_1$ indicates the orthogonality of $\widehat{\mathcal{M}}_{\pm}$ with respect to $(\cdot,\cdot)_1$.

The Krein space structure of  $\widehat{\sH}$ can be introduced by \eqref{fr30}.
Considering \eqref{fr30} as the fundamental decomposition of $\widehat{\sH}$, we define the new indefinite inner
product
\begin{equation}\label{fr15}
[f,g]_1:=[f_{\widehat{\mathcal{M}}_{+}}, g_{\widehat{\mathcal{M}}_{+}}]+[f_{\widehat{\mathcal{M}}_{-}}, g_{\widehat{\mathcal{M}}_{-}}],  \quad f,g\in\widehat{\sH}.
\end{equation}

The Hilbert space associated with the Krein space $(\widehat{\sH}, [\cdot,\cdot]_1)$ coincides with
$(\widehat{\sH}, (\cdot,\cdot)_1)$. The corresponding operator of fundamental symmetry
$J_{\widehat{\mathcal{M}}}$  is the closure in $(\widehat{\sH}, (\cdot,\cdot)_1)$ of the operator $J_{\mathcal{M}}$ defined by \eqref{fr10b}.
The operator $J_{\widehat{\mathcal{M}}}$  is self-adjoint in $(\widehat{\sH}, [\cdot,\cdot]_1)$
and  $(\cdot, \cdot)_1=[J_{\widehat{\mathcal{M}}}\cdot,\cdot]_1$.

On the other hand, the operator $J_{\mathcal{M}}$  defined by \eqref{fr10b} is a closed unbounded operator in $(\sH, (\cdot, \cdot))$.
The adjoint $J_{\mathcal{M}}^+$ of $J_{\mathcal{M}}$ with respect to $[\cdot, \cdot]$ in $\mathfrak{H}$
has the domain ${\mathcal D}(J_{\mathcal{M}}^+)=\mathcal{M}_-^{[\perp]}\dot{+}\mathcal{M}_+^{[\perp]}$
and its action is defined similarly to \eqref{fr10b}, where the maximal definite subspaces  $\mathcal{M}_\mp^{[\perp]}$
are used instead of $\mathcal{M}_\pm$.  In general, the domains ${\mathcal D}(J_{\mathcal{M}})$
and ${\mathcal D}(J_{\mathcal{M}}^+)$ are not coincide.  Therefore, the operator
$\cC$ in \eqref{fr14} can only be defined on $\mathcal{D}_0={\mathcal D}(J_{\mathcal{M}})\cap{\mathcal D}(J_{\mathcal{M}}^+)$.
By analogy with the proof of Lemma \ref{lem2}, we decide that
\begin{equation}\label{fr59}
(f,g)_1=[\cC{f}, g],  \qquad \forall{f\in\mathcal{D}_0}, \quad \forall{g}\in{\mathcal D}.
\end{equation}

{\bf III.} For the important particular case where $\mathcal{M}_\pm$ are \emph{$J$-orthogonal},
the subspace $\mathcal{D}_0$ coincides with  $\mathcal{D}$  and the operator $\cC$ is $J$-self-adjoint.
Moreover, $\cC$ coincides  with $J_{\mathcal{M}}$ and hence, it
is characterized by the additional condition $\cC^2f=f$ for all $f\in\mathcal{D}(\cC)=\mathcal{D}(J_{\mathcal{M}})$.
In this  case,  the operator $\cC$  can be presented as $\cC=Je^Q$, where $Q$  is an
unbounded self-adjoint  operator in $(\sH, (\cdot,\cdot))$  anticommuting with $J$ \cite[Theorem 2.1]{Sud}.
It follows from above that $e^Q=J\cC=JJ_{\mathcal{M}}$.  Therefore, $e^{Q/2}=\sqrt{JJ_{\mathcal{M}}}$.

The operator $Q$ characterizes the `deviation' of subspaces $\mathcal{M}_\pm$ with respect to the subspaces $\sH_\pm$ in the fundamental decomposition
\eqref{fr7}.  Precisely \cite{Sud},
\begin{equation}\label{fr101}
\mathcal{M}_+=(I-\tanh\frac{Q}{2})\sH_+,  \qquad  \mathcal{M}_-=(I-\tanh\frac{Q}{2})\sH_-.
\end{equation}

The modulus of a self-adjoint contraction operator $\tanh{Q}/{2}$ can be related with operator angles
$\Theta(\sH_\pm, \mathcal{M}_\pm)$  between $\sH_{\pm}$  and  $\mathcal{M}_{\pm}$,
\cite[Section 2.3]{Sud}:
$$
|\tanh\frac{Q}{2}|=\tan\Theta(\sH_+, \mathcal{M}_+)P_+ + \tan\Theta(\sH_-, \mathcal{M}_-)P_-.
$$
Here, the operator angles $\Theta(\sH_{\pm},  \mathcal{M}_\pm)$  are determined as follows:
$$
\Theta(\sH_\pm, \mathcal{M}_\pm):=\arcsin\sqrt{I - P_{\pm}P_{\mathcal{M}_\pm}}\upharpoonright_{\sH_{\pm}},
$$
where $P_{\pm}$, $P_{\mathcal{M}_{\pm}}$ are orthogonal projections in $\sH$
onto $\sH_{\pm}$  and  $\mathcal{M_{\pm}}$, respectively.

Let $\mathfrak{D}$ be the \emph{energetic space} constructed by the self-adjoint operator $e^Q$.  In other words,
$\mathfrak{D}$ denotes the completion of $\mathcal{D}=\mathcal{D}(\cC)=\mathcal{D}(e^Q)$  with respect to the energetic norm
$$
\|f\|^2_{en}:=\|f\|^2+\|f\|_1^2=\|f\|^2+[{\cC}f, f]=\|f\|^2+({e^Q}f, f),  \qquad  f\in\mathcal{D}.
$$
The energetic space $\mathfrak{D}$ coincides with $\mathcal{D}(e^{Q/2})$ and it is a Hilbert space
$(\mathfrak{D}, (\cdot,\cdot)_{en})$
with respect to the energetic scalar product $
(\cdot, \cdot)_{en}=(\cdot, \cdot)+(e^{Q/2}\cdot, e^{Q/2}\cdot)$. 

Sometimes, in the sequel, we will consider $\mathfrak{D}$ as the set of elements (without topology).
In this context, the term \emph{energetic linear manifold} will be used.

Comparing the definitions of $\widehat{\sH}$ and $\mathfrak{D}$ leads to the conclusion that  \emph{the energetic linear manifold  $\mathfrak{D}$ coincides
with the common part of $\sH$ and $\widehat{\sH}$,}  i.e.,   $\mathfrak{D}=\sH\cap\widehat{\sH}$.  Obviously, the formula \eqref{fr59} can be extended
onto $\mathfrak{D}$ as follows:
\begin{equation}\label{fr59b}
(f,g)_1=(e^{Q/2}{f}, e^{Q/2}g),  \qquad \forall{f, g}\in{\mathfrak D}.
\end{equation}

The Hilbert space $\widehat{\sH}$ cannot be interpreted as a Hilbert space $\sH_W$ with $W$-metric considered in \cite[Section 2.1]{new3}.  Indeed, 
if  $\widehat{\sH}=\sH_W$ for some self-adjoint $W$, then $\mathfrak{D}=\sH$ (because $\sH_W$ is the completion of $\sH$) 
that is impossible since $e^{Q/2}$ is an unbounded operator.  

\begin{lemma}\label{agga20}
If $\mathcal{M}_\pm$ are $J$-orthogonal, then the indefinite inner products
$[\cdot, \cdot]$  and $[\cdot, \cdot]_1$ coincide on the energetic linear manifold $\mathfrak{D}$.
In the Hilbert space $(\widehat{\sH}, (\cdot,\cdot)_1)$,  the indefinite inner product $[\cdot, \cdot]$
(defined originally on $\mathfrak{D}$) can be continuously extended onto $\widehat{\sH}$ and this extension coincides with
  $[\cdot, \cdot]_1$.
\end{lemma}
\begin{proof}
By virtue of  \eqref{fr15}  and the $J$-orthogonality of $\mathcal{M}_\pm$,
$$
[f,g]=[f_{\mathcal{ M}_+}+f_{\mathcal{ M}_-},  g_{\mathcal{ M}_+}+g_{\mathcal{ M}_-}]=[f_{\mathcal{ M}_+}, g_{\mathcal{ M}_+}]+[f_{\mathcal{ M}_-},g_{\mathcal{ M}_-}]=[f, g]_1, \ \forall f,g\in\mathcal{D}.
$$
The obtained relation are extended onto $\mathfrak{D}$ by the continuity because $|[f,f]|\leq\|f\|^2$ and
 $|[f,g]_1|\leq\|f\|^2_1$.
 The second statement of Lemma follows from the coincidence of $[\cdot,\cdot]$ and
 $[\cdot,\cdot]_1$ on $\mathfrak{D}$ and the definition of $(\widehat{\sH}, (\cdot,\cdot)_1)$.
\end{proof}

\section{Frames in Krein spaces.  Reconstruction formulas}
Since the subspaces $\mathcal{M}_{\pm}$  in Definition \ref{def3} are assumed to be maximal definite, the
 linear manifold $\mathcal{D}$ in \eqref{sas12} is a dense set in  $(\sH, (\cdot,\cdot))$.
Properties of the corresponding $J$-frame $\cF$ depend
on the fact does $\mathcal{D}$ coincide with $\sH$ or not.

\subsection{The set $\mathcal{D}$  coincides with $\sH$.}\label{3.1}
\begin{proposition}\label{prop1}
Let $\cF$ be a $J$-frame in the sense of Definition \ref{def3} and $\mathcal{D}=\sH$.
Then $\cF$ is a conventional frame in the Hilbert space $(\sH, (\cdot,\cdot)_1)$ with the same frame bounds
$A\leq B$. Moreover, $\cF$ is $J$-frame  in the sense of Definition \ref{def2}.
\end{proposition}
\begin{proof}
If $\mathcal{D}=\sH$, then the subspaces $\mathcal{M}_{\pm}$  are maximal uniformly definite in the Krein space $(\sH, [\cdot,\cdot])$ and \eqref{sas12} coincides with \eqref{sas12b}.
The decomposition \eqref{sas12b} determines the new inner product  $(\cdot,\cdot)_1$ on $\sH$  which is equivalent to
 $(\cdot,\cdot)$.

 By the construction, $|[f,g]|=|(f, g)_1|$ where $f,g \in\mathcal{M}_{+}$ or $f,g \in\mathcal{M}_{-}$.
 Moreover, the subspaces $\mathcal{M}_\pm$ are orthogonal with respect to $(\cdot, \cdot)_1$. Therefore, 
 the inequalities \eqref{fr5b}  can be rewritten as
\begin{equation}\label{agga22}
A\|f\|^2_1\leq \sum_{n\in\mathbb{N}}{|(f,f_n)_1|^2}\leq B\|f\|^2_1, \quad
 \forall{f}\in\mbox{span}\{\cF\}.
\end{equation}
where $\mbox{span}\{\cF\}=\mbox{span}\{\cF_-\}\dot{+}\mbox{span}\{\cF_+\}$ is a dense set  in $\sH$.
The inequality \eqref{agga22} holds true for any $f\in\sH$
due to \cite[Lemma 5.1.7]{Ole}. Therefore, $\cF$ is a frame
in the Hilbert space $(\sH, (\cdot,\cdot)_1)$ with the same frame bounds $A\leq B$.

The concept of  $J$-frame defined in Definition \ref{def2}
corresponds to the case where $\mathcal{M}_{\pm}$ are maximal uniformly definite subspaces.
Therefore,  the $J$-frame $\cF$ considered above is also a $J$-frame in the sense of Definition
\ref{def2}.
\end{proof}

By  Proposition \ref{prop1}, the synthesis operator $T : l_2(\mathbb{N}) \to \sH$ associated to $J$-frame $\cF$ is well defined.
Denote by $T^\dag$ its adjoint as an operator mapping of the Krein space $(\sH, [\cdot,\cdot])$ into
the Hilbert space $l_2(\mathbb{N})$ i.e., $[T\{c_n\}, f]=(\{c_n\}, T^\dag{f})_{l_2(\mathbb{N})}$, where $\{c_n\}\in{l_2(\mathbb{N})}$,
$f\in\sH$. It is easy to verify that  $T^\dag{f}=\{[f,f_n]\}$.

The operator $S : \sH \to \sH$
\begin{equation}\label{fr17}
Sf=TT^\dag{f}=\sum_{n=1}^\infty[f,f_n]f_n
\end{equation}
is called \emph{a $J$-frame operator associated to $J$-frame} $\cF$.
\begin{proposition}\label{fr37}
The $J$-frame operator $S$ is a  $J$-positive, $J$-self-adjoint bounded operator in the Krein space
$(\sH, [\cdot,\cdot])$ and $0\in\rho(S)$. The reconstruction formula holds
\begin{equation}\label{fr19}
f=\sum_{n=1}^\infty[f,S^{-1}f_n]f_n=\sum_{n=1}^\infty{[f,f_n]S^{-1}f_n},
\end{equation}
where the series converge in the Hilbert space $\sH$.
\end{proposition}
\begin{proof}
By Proposition \ref{prop1}, the $J$-frame $\cF$ is a conventional frame in the Hilbert
space $(\sH, (\cdot,\cdot)_1)$. The corresponding frame operator $S_1$ has the form
\begin{equation}\label{fr41b}
S_1f=\sum_{n=1}^\infty(f,f_n)_{1}f_n
\end{equation}
and it is a positive self-adjoint operator in $(\sH, (\cdot,\cdot)_1)$ such that $0\in\rho(S_1)$.
By virtue of \eqref{fr14} and \eqref{fr17},
\begin{equation}\label{fr20}
S_{1}=S\cC, \qquad S^{-1}=\cC{S_1^{-1}}.
\end{equation}

It follows from \eqref{fr20} that  $S^{-1}$ and  $S$ are  $J$-positive, $J$-self-adjoint bounded operators in the Krein space
$(\sH, [\cdot,\cdot])$.
The reconstruction formula for the frame $\cF$ in the Hilbert space $(\sH, (\cdot,\cdot)_1)$  is:
\begin{equation}\label{fr61}
f=\sum_{n=1}^\infty{(f,S^{-1}_1f_n)_1f_n}=\sum_{n=1}^\infty{(f,f_n)_1S^{-1}_1f_n}
\end{equation}
and its first part is transformed to the first part of \eqref{fr19} since
$$
(f,S^{-1}_1f_n)_1=[\cC{f}, S^{-1}_1f_n]=[f, {\cC}S^{-1}_1f_n]=[f, S^{-1}f_n].
$$

Multiplying  \eqref{fr61} by  the operator $\cC$ and using the second relation in \eqref{fr20} we obtain
$$
\cC{f}=\sum_{n=1}^\infty{(f,f_n)_1S^{-1}f_n}=\sum_{n=1}^\infty{[\cC{f},f_n]_1S^{-1}f_n}
$$
that is equivalent to the second relation in \eqref{fr19} (since $\mathcal{R}(\cC)=\sH$). 
Thus, the series \eqref{fr19} are convergent in $\sH$ with respect to $(\cdot,\cdot)_1$. 
Obviously, they remain convergent with  respect to $(\cdot, \cdot)$ since the inner products
$(\cdot, \cdot)$ and $(\cdot, \cdot)_1$ are equivalent.
\end{proof}

\begin{remark}\label{rem1}  In \cite{Giribet2}, the Hilbert space $l_2(\mathbb{N})$
was considered as a Krein space with indefinite metric generated by the fundamental decomposition
 \eqref{fr28} and the adjoint $T^+$ of the synthesis operator $T$  was calculated as an operator
 acting between \emph{two Krein spaces $\sH$ and $l_2(\mathbb{N})$},  that is:
 $[T\{c_n\}, f]=[\{c_n\}, T^+f]_{l_2(\mathbb{N})}$, where  $[\{c_n\}, \{c_n\}]_{l_2(\mathbb{N})}=\sum_{\mathbb{N}}\sigma_n|c_n|^2$ and
 $\sigma_n=\mbox{sgn}[f_n,f_n]$. It is easy to check that
 $$
 T^+f=\{\sigma_n[f,f_n]\}.
 $$
 Then, the corresponding $J$-frame operator $\widetilde{S}=TT^+$ acts as, cf. \eqref{fr17}
$$
\widetilde{S}=\sum_{n=1}^\infty\sigma_n[f,f_n]f_n
$$
and its detailed investigation can be found in \cite{Giribet4}.
By virtue of \eqref{fr10b} and \eqref{fr17}, the $J$-frame operators $S$ and $\widetilde{S}$ are related as follows:
\begin{equation}\label{fr21}
J_{\mathcal{M}}S=\widetilde{S}, \qquad  SJ_{\mathcal{M}}^+=\widetilde{S}.
\end{equation}
Hence, $0\in\rho(\widetilde{S})$ and $\widetilde{S}$ is a $J$-self-adjoint operator in
$(\sH, [\cdot,\cdot])$ since $\widetilde{S}^+=(J_{\mathcal{M}}S)^+=SJ_{\mathcal{M}}^+=\widetilde{S}$.
However, in contract to $S$, the $J$-frame operator $\widetilde{S}$ cannot be  $J$-positive.
Indeed, $
[\widetilde{S}f, f]=-\sum_{n\in\mathbb{N}_-}|[f,f_n]|^2<0
$ for $f\in\mathcal{M}_+^{[\perp]}$.

By virtue of \eqref{fr21} the reconstruction formula \eqref{fr19} takes the form:
$$
f=\sum_{n=1}^\infty\sigma_n[f,\widetilde{S}^{-1}f_n]f_n=\sum_{n=1}^\infty\sigma_n[f,f_n]\widetilde{S}^{-1}f_n.
$$
\end{remark}

\subsection{The set $\mathcal{D}$ does not coincide with $\sH$.}
\begin{proposition}\label{prop3}
Let $\cF$ be a $J$-frame in the sense of Definition \ref{def3} and 
 $\mathcal{D}\not=\sH$.
Then  the  $J$-frame $\cF$ is a conventional frame in the new Hilbert space $(\widehat{\sH},  (\cdot,\cdot)_1)$ with the same frame bounds
$A\leq B$.
\end{proposition}
\begin{proof}
If $\mathcal{D}\not=\sH$,  then at least one of subspaces $\mathcal{M}_{\pm}$
is maximal definite but no uniformly definite.  In this case, the
 direct sum \eqref{sas12} generates the new  Hilbert space $(\widehat{\sH}, (\cdot,\cdot)_1)$.
Similarly to the proof of Proposition \ref{prop1} we rewrite the inequalities \eqref{fr5b} as \eqref{agga22},  where
$\mbox{span}\{\cF\}$ is a dense set in $\widehat\sH$ and extend \eqref{agga22} onto $\widehat\sH$. 
Therefore,  $\cF$ is a frame in the Hilbert space $(\widehat{\sH}, (\cdot,\cdot)_1)$ with the same frame bounds $A\leq B$.
\end{proof}

By virtue of Proposition \ref{prop3}, the reconstruction formula \eqref{fr61} is true for all $f\in\widehat{\sH}$
(the series converge in the Hilbert space $(\widehat{\sH}, (\cdot,\cdot)_1)$).  However, in contrast to Section \ref{3.1},
we cannot transform \eqref{fr61} to  \eqref{fr19}.  The matter is that  the
inner products  $(\cdot, \cdot)$ and $(\cdot, \cdot)_1$ are not equivalent and we cannot express  $(\cdot, \cdot)_1$
via the indefinite inner product $[\cdot,\cdot]$ on $\widehat{\sH}$.
Nevertheless, if
$f\in\mathcal{D}_0={\mathcal D}(J_{\mathcal{M}})\cap{\mathcal D}(J_{\mathcal{M}}^+)$,
the formula \eqref{fr59} allows one to rewrite the coefficients $(f,f_n)_{1}$ of \eqref{fr61}  in terms of $[\cdot,\cdot]$:
$(f,f_n)_{1}=[\cC{f}, f_n].$ 

\subsection{The subspaces $\mathcal{M}_\pm$ are $J$-orthogonal.}
The above results are essentially simplified when the subspaces $\mathcal{M}_\pm$ are $J$-orthogonal.
\begin{proposition}\label{fr41}
Let $\cF$ satisfy the assumptions of Proposition \ref{prop3} and let the subspaces $\mathcal{M}_\pm$ be $J$-orthogonal.
Then, for all elements  $f$ from the energetic linear manifold $\mathfrak{D}$ (see Section \ref{2.2}),
\begin{equation}\label{fr71}
f=\sum_{n=1}^\infty{[f,f_n]S^{-1}f_n},
\end{equation}
where $Sf=\sum_{n=1}^\infty[f,f_n]f_n$ and the series converge in the Hilbert space $(\widehat{\sH}, (\cdot,\cdot)_1)$.
\end{proposition}
\begin{proof}
By the construction, the indefinite inner products $[\cdot,\cdot]$ and $[\cdot,\cdot]_1$ coincide on ${\mathcal{M}}_{\pm}$ and $\mbox{span}\{\cF_\pm\}$
are dense subsets of the Hilbert spaces $({\mathcal{M}}_{\pm}, \pm[\cdot,\cdot])$ and $(\widehat{{\mathcal{M}}}_{\pm}, \pm[\cdot,\cdot]_1)$.
This means that we can replace,  in  Definition \ref{def3},  $[\cdot,\cdot]$ and $\mathcal{M}_\pm$ by  $[\cdot,\cdot]_1$ and  $\widehat{\mathcal{M}}_\pm$, respectively.  
Therefore, \emph{the $J$-frame $\cF$ in the Krein space $({\sH}, [\cdot,\cdot])$
 is simultaneously  a  $J_{\widehat{\mathcal{M}}}$-frame in the Krein space  $(\widehat{\sH}, [\cdot,\cdot]_1)$}.
 In the latter case,  since decomposition  \eqref{fr30} holds,
we can apply Proposition \ref{fr37} for  the $J_{\widehat{\mathcal{M}}}$-frame $\cF$.
Hence, for all $f\in\widehat{\sH}$,
\begin{equation}\label{fr38}
f=\sum_{n=1}^\infty[f,S^{-1}f_n]_{1}f_n=\sum_{n=1}^\infty{[f,f_n]_{1}S^{-1}f_n},
\end{equation}
where  $Sf=\sum_{n=1}^\infty[f,f_n]_1f_n$  is a  $J_{\widehat{\mathcal{M}}}$-self-adjoint  bounded operator in the Krein space
$(\widehat{\sH}, [\cdot,\cdot]_1)$.
The series \eqref{fr38} converge in the Hilbert space  $(\widehat{\sH}, (\cdot,\cdot)_1)$.  In particular,  for
 $f\in\mathfrak{D}$  we can  change $[f,f_n]_1$ into $[f,f_n]$  due to Lemma \ref{agga20}.
\end{proof}
\begin{remark}
 Additional assumption that $S^{-1}f_n\in\mathfrak{D}$  and Lemma \ref{agga20} lead to the conclusion
 that
$f=\sum_{n=1}^\infty[f,S^{-1}f_n]f_n$  for all $f\in\mathfrak{D}.$
\end{remark}

\begin{proposition}\label{fr81}
Let $\cF=\{f_n\}$  be a  $J$-frame and  let the subspaces $\mathcal{M}_\pm$ be $J$-orthogonal.
Then there exists a self-adjoint operator $Q$ in the Hilbert space $(\sH, (\cdot,\cdot))$, which anticommutes with
$J$ and such that the sequence $\{e^{Q/2}f_n\}$ is a conventional frame in $(\sH, (\cdot,\cdot))$ with the same frame bounds
$A\leq B$.
\end{proposition}
\begin{proof}
Each direct sum of $J$-orthogonal maximal definite subspaces ${\mathcal{M}_{\pm}}$  generates
 a self-adjoint operator $Q$ in the Hilbert space $\sH$, which anticommutes with
$J$ (see Section \ref{2.2}).  In this case, relation \eqref{fr59b} holds
for any $f,g\in\mathfrak{D}$.

Let us assume that the direct sum of ${\mathcal{M}_{\pm}}$ does not coincide with $\sH$. Then,
according to Proposition \ref{prop3},  $\cF$ is a  frame in the Hilbert space $(\widehat{\sH}, (\cdot,\cdot)_1)$.
Therefore, $A(f,f)_1\leq \sum_{n\in\mathbb{N}}{|(f,f_n)_1|^2}\leq B(f, f)_1$  for all   $f\in\mathfrak{D}$. Using \eqref{fr59b},
we rewrite these inequalities as follows: $A\|\gamma\|^2\leq \sum_{n\in\mathbb{N}}{|(\gamma, e^{Q/2}f_n)|^2}\leq B\|\gamma\|^2$,
where $\gamma=e^{Q/2}f$ runs the dense set $\mathcal{R}(e^{Q/2})$ in $(\sH, (\cdot, \cdot))$.
The obtained inequalities can be extended onto $\sH$ due to \cite[Lemma 5.1.7]{Ole}. Therefore, $\{e^{Q/2}f_n\}$ is a frame in $(\sH, (\cdot,\cdot))$.
The case ${\mathcal{M}_{+}}\dot{+}{\mathcal{M}_{-}}=\sH$ is considered in the same manner with the use of
Proposition \ref{prop1}.
\end{proof}

The inverse statement is also true.
\begin{proposition}\label{fr91}
Assume that $\{g_n\}$ is a frame in the Hilbert space $(\sH, (\cdot,\cdot))$ such that
each $g_n$ belongs to one of the subspaces $\sH_+$ or  $\sH_-$ of the fundamental decomposition
\eqref{fr7}  and there exists a self-adjoint operator $Q$  in $(\sH, (\cdot,\cdot))$ which anticommutes with $J$ and such that
 $\{\cosh{Q/2} \ g_n\}$  is a complete set in $\sH$. Then the sequence $\{e^{-Q/2}g_n\}$ is a $J$-frame.
\end{proposition}
\begin{proof}
The operator $\cosh{Q/2}=\frac{1}{2}(e^{Q/2}+e^{-Q/2})$  commutes with $J$
(since $Q$ anticommutes with $J$).  Therefore, the vector
$x_n=\cosh{Q/2} \ g_n$ belongs to the  same subspace ($\sH_+$ or  $\sH_-$)  that $g_n$.
Denote $x_n^{\pm}=x_n$ if $x_n\in\sH_{\pm}$.

According to Section \ref{2.2},  the operator $Q$
uniquely determines a $J$-orthogonal pair of maximal definite subspaces $\mathcal{M}_{\pm}$,
see \eqref{fr101}. Denote
$$
\mathcal{M}_+'=\overline{\mbox{span}\{(I-\tanh{Q/2})x_n^{+}\}},  \qquad \mathcal{M}_-'=\overline{\mbox{span}\{(I-\tanh{Q/2})x_n^{-}\}}.
$$

In view of  \eqref{fr101} and the fact that the set $\{x_n=\cosh{Q/2} \ g_n\}$
is complete in $\sH$,  we decide that  $\mathcal{M}_{\pm}'=\mathcal{M}_\pm$.
 Moreover,
$$
f_n=(I-\tanh\frac{Q}{2})x_n=(I-\tanh\frac{Q}{2})\cosh\frac{Q}{2} \ g_n=\cosh\frac{Q}{2} \ g_n - \sinh\frac{Q}{2} \ g_n=e^\frac{-Q}{2}g_n.
$$
Therefore,  $g_n=e^{Q/2}f_n$, where $f_n$ belongs to one of the sets  $\mathcal{M}_\pm$
and  $\mbox{sgn}[g_n,g_n]=\mbox{sgn}[f_n,f_n]$.
The frame inequalities  for the frame $\{g_n\}$ can be rewritten as follows:
\begin{equation}\label{fr103}
A\|\gamma\|^2\leq \sum_{n\in\mathbb{N}}{|(\gamma, e^{Q/2}f_n)|^2}\leq B\|\gamma\|^2,  \qquad \gamma\in\sH.
\end{equation}
Assuming in \eqref{fr103} that $\gamma=e^{Q/2}f$, where $f\in\mathcal{M}_\pm$ and using \eqref{fr9b}, \eqref{fr59b} we obtain
the inequalities \eqref{fr5b} for all ${f}\in\mathcal{M}_\pm$.
Therefore, $\{f_n=e^{-Q/2}g_n\}$ is a $J$-frame in the sense of Definition \ref{def3} .
\end{proof}

\section{Frames in Krein spaces. Selected topics}
\subsection{$A$-tight frames and exact frames.}
A $J$-frame $\cF$ is called \emph{$A$-tight}, if $A=B$ in  Definition \ref{def3}.
\begin{proposition}\label{prop4}
If a $J$-frame $\cF$ is $A$-tight, then for all $f\in\mathcal{D}$
$$
f=\frac{1}{2A}\sum_{n=1}^\infty[(J_{\mathcal{M}}+\sigma_nI)f, f_n]f_n, \qquad \sigma_n=\mbox{sgn}[f_n,f_n].
$$
The series converges in the Hilbert space $(\sH, (\cdot,\cdot))$ when $\mathcal{D}=\sH$. Otherwise ($\mathcal{D}\not=\sH$),
the convergence should be considered in $(\widehat{\sH}, (\cdot,\cdot)_1)$.
\end{proposition}
\begin{proof}
Due to Propositions \ref{prop1},  \ref{prop3},  $\cF$  is  a conventional $A$-tight  frame in one of the Hilbert spaces
$\sH$ or $\widehat{\sH}$. Then the frame operator $S_1$ in \eqref{fr41b} coincides with $AI$  \cite[Theorem 8.3]{Heil}.
According to \eqref{fr61}, $f=\frac{1}{A}\sum_{n=1}^\infty(f, f_n)_1f_n$  for all
$f\in\mathcal{D}$.

It follows from the proof of Lemma \ref{lem2} and \eqref{fr10b},
$$
(f, f_n)_1=\frac{1}{2}([J_{\mathcal{M}}f, f_n]+[f, J_{\mathcal{M}}f_n])=\frac{1}{2}[(J_{\mathcal{M}}+\sigma_nI)f, f_n]
$$
that completes the proof.
\end{proof}

Let $\cF$ be a $J$-frame. Then $\cF$ turns out to be a frame either in the Hilbert space
 $\sH$ or in the Hilbert space  $\widehat{\sH}$.
 We will say  that $\cF$ is an \emph{exact $J$-frame} if it is an exact frame in one of the Hilbert spaces above.

If ${\mathcal D}=\sH$,  then an exact  $J$-frame  $\cF$  turns out to be a Riesz basis in the Hilbert space
$\sH$  (see,  Proposition \ref{prop1} and \cite[Theorem 8.27]{Heil}).  Similarly, if ${\mathcal D}\not=\sH$, then
$\cF$ is a Riesz basis in $\widehat{\sH}$ but it cannot be a basis in $\sH$.
We can only state that the sets  $\cF_{\pm}$ defined in \eqref{fr101b}  are  \emph{exact sequences}  (i.e., they are minimal and complete)
in the subspaces  $(\mathcal{M}_{\pm}, (\cdot,\cdot))$ of the Hilbert space $\sH$.
 Indeed, the completeness of $\cF_{\pm}$ in $\mathcal{M}_{\pm}$  follows from Definition \ref{def3}.
   Assume that $\cF_+=\{f_n\}_{n\in\mathbb{N}_+}$ is not a minimal sequence in
 $(\mathcal{M}_{+}, (\cdot,\cdot))$. Then there exist $f\in\cF_+$ and linear combinations 
 $g_k$ of $\cF_+\setminus\{f\}$
such that $\|f-g_k\|\to{0}$. Therefore $\|f-g_k\|_1=\sqrt{[f-g_k, f-g_k]}$ also vanishes when $k\to\infty$ (because $\|\cdot\|^2\geq[\cdot,\cdot]=(\cdot,\cdot)_1$ on $\mathcal{M}_+$).
The latter means that $\cF_+$  cannot be a minimal sequence in the Hilbert space $(\widehat{\mathcal{M}}_+,  (\cdot,\cdot)_1)$ that is impossible.
 Therefore, $\cF_+$  is a minimal sequence in  $(\mathcal{M}_{+}, (\cdot,\cdot))$.  The case $\cF_-$ is considered in a similar manner.

\subsection{$J$-orthogonal sequences and $J$-orthogonal Schauder bases.}

{\bf I.}   Let a sequence  $\{f_n\}_{n=1}^{\infty}$  be
$J$-orthogonal (i.e.,  $[f_n, f_m]=0$ for $n\not=m$). Then the subspaces
$\mathcal{M}_{\pm}$ defined by \eqref{fr101b} are $J$-orthogonal
in the Krein space $(\mathfrak{H}, [\cdot,\cdot])$.

\begin{proposition}\label{prop5}
An $J$-orthogonal sequence $\cF=\{f_n\}$ is  a $J$-frame in the sense of Definition \ref{def3} if and only if $\cF$ is $J$-bounded, i.e.,
\begin{equation}\label{fr299}
0<A\leq |[f_n, f_n]| \leq B<\infty, \qquad \forall f_n\in\cF.
\end{equation}
  and the subspaces  ${\mathcal M}_\pm$ are maximal definite in the Krein space
$(\sH, [\cdot,\cdot])$.
 \end{proposition}
\begin{proof}
The $J$-orthogonality of $\cF$ and \eqref{fr299} allow one
to verify directly \eqref{fr5b}. Therefore $J$-boundedness of $\cF$ means that $\cF$ is a $J$-frame with frame bounds $A\leq{B}$.
Conversely, if $\cF$ is a $J$-frame, then  \eqref{fr5b} with $f=f_n$ gives \eqref{fr299}.
\end{proof}

\begin{lemma}
The subspaces ${\mathcal M}_\pm$ are maximal definite if and only if the operator
  $JJ_{\mathcal{M}}$  is self-adjoint in the Hilbert space $(\sH, (\cdot,\cdot))$.
\end{lemma}
\begin{proof}
If ${\mathcal M}_\pm$ are maximal definite, then $JJ_{\mathcal{M}}$ coincides
with the self-adjoint operator $e^{Q}$ from Section \ref{2.2}.
Conversely, if  $JJ_{\mathcal{M}}$ is  self-adjoint, then ${\mathcal{M}_{\pm}}$
have to be maximal (it follows from \cite[Proposition 4.2]{Sud} since $JJ_{\mathcal{M}}$ coincides with $G_0$).
\end{proof}

\begin{remark} The maximality of ${\mathcal{M}_{\pm}}$ ensures the completeness
of the corresponding $J$--orthogonal sequence $\{f_n\}$ in $(\sH, (\cdot,\cdot))$. The inverse statement
is not true (see, e.g., \cite{Sud}).  We can only state that the subspaces ${\mathcal{M}_{\pm}}$
are definite in   $(\sH, [\cdot,\cdot])$ and theirs direct sum is dense in $(\sH, (\cdot, \cdot))$.
\end{remark}

Let $\{f_n\}$ be a complete  $J$-orthogonal sequence. Its biorthogonal sequence $\{\gamma_n\}$  consists of  the elements
 $\gamma_n=\frac{Jf_n}{[f_n,f_n]}.$ By Proposition \ref{prop5}, if one of sequences $\{f_n\}$,   $\{\gamma_n\}$ is
  a $J$-frame then another one is also a $J$-frame.

Let a complete $J$-orthogonal sequence  $\cF=\{f_n\}$ be a $J$-frame.  Then Proposition \ref{fr41} holds and the operator
$Sf=\sum_{n=1}^\infty[f,f_n]f_n$ is well defined on the energetic linear manifold $\mathfrak{D}$. It is easy to see
that $S^{-1}f_n={f_n}/{[f_n,f_n]}$ and therefore, the decomposition \eqref{fr71}
can be  rewritten as follows:
\begin{equation}\label{fr51}
f=\sum_{n=1}^\infty\frac{[f,f_n]}{[f_n,f_n]}f_n=\sum_{n=1}^\infty(f,\gamma_n)f_n, \qquad f\in\mathfrak{D},
\end{equation}
where the series  \eqref{fr51} converges  with respect to the norm
$\|\cdot\|_1=\sqrt{(e^{Q}\cdot,\cdot)}$, where $e^{Q}$ is a positive self-adjoint operator
in $\sH$ such that
$
e^{Q}f_n=JJ_{\mathcal{M}}{f_n}=\sigma_n{J}f_n=\sigma_n[f_n,f_n]\gamma_n=|[f_n,f_n]|\gamma_n.
$

\begin{corollary}
Let a complete $J$-orthogonal sequence $\cF=\{f_n\}$ be a $J$-frame. Then $\cF$ is a Bessel sequence in $\sH$ if and only
if  $\cF$ is a Riesz basis in $\sH$.
\end{corollary}
\begin{proof}
If a complete $J$-orthogonal set  $\cF=\{f_n\}$ is a $J$-frame, then
the sequence $\{f_n\}$ is a bounded orthogonal basis in $(\widehat{\sH}, (\cdot,\cdot)_1)$.
Therefore, each $f\in\widehat{\sH}$ admits the presentation $f=\sum_{n\in\mathbb{N}}c_nf_n$, where
$\{c_n\}\in{l_2(\mathbb{N})}$.
According to \cite[Theorem 7.4]{Heil},  the property of being Bessel sequence for  $\{f_n\}$ is equivalent to the fact
that  the series $\sum_{n\in\mathbb{N}}c_n{f_n}$ converges in $({\sH}, (\cdot,\cdot))$
for each sequence $\{c_n\}$  from ${l_2(\mathbb{N})}$.
The latter means that $\widehat{\sH}=\sH$ and  therefore,  $\cF$ is a Riesz basis in $\sH$.
\end{proof}

{\bf II.} In what follows we  suppose  that  $\cF=\{f_n\}$
is a Schauder basis of $(\sH, (\cdot,\cdot))$.

\begin{corollary}\label{fr73}
An $J$-orthogonal Schauder basis $\cF$  turns out to be a $J$-frame  if and only  if ${\cF}$ is $J$-bounded.
\end{corollary}
\begin{proof}
It follows from Proposition \ref{prop5} and the fact that the subspaces  ${\mathcal M}_\pm$ are maximal
in the case of  an $J$-orthogonal Schauder basis  \cite{AZ}.
\end{proof}
\begin{corollary}\label{fr72}
 If an $J$-orthogonal Schauder basis  $\cF=\{f_n\}$ is a $J$-frame,  then
 the series \eqref{fr51} is convergent in the energetic space $(\mathfrak{D}, (\cdot,\cdot)_{en})$, i.e.,
 it is convergent with respect to the energetic norm $\|\cdot\|_{en}$.
\end{corollary}
\begin{proof}
It follows from the definition of energetic norm in Section \ref{2.2} and the convergence of
\eqref{fr51} with respect to $\|\cdot\|$ (since $\cF$ is 
 Schauder basis).
\end{proof}

 The energetic linear manifold $\mathfrak{D}$ can be easily described.

\begin{proposition}\label{fr78}
Let an $J$-orthogonal Schauder basis  $\cF=\{f_n\}$ be a $J$-frame and $f\in\sH$. Then 
$f\in\mathfrak{D}$ if and only if  the sequence $\{[f,f_n]\}$  belongs to $l_2(\mathbb{N})$.
\end{proposition}
\begin{proof}
If $f\in\mathfrak{D}$, then \eqref{fr51} converges simultaneously in $\sH$ and $\widehat{\sH}$.
Moreover, \eqref{fr299} holds and $\{f_n\}$ is a bounded orthogonal basis in $\widehat{\sH}$. 
This means that the sequence $\{[f,f_n]\}$ belongs to $l_2(\mathbb{N})$. 

To prove the inverse implication, we note that
 \eqref{fr51} converges for all $f\in\sH$  (since $\cF$ is a Schauder basis).
If $\{[f,f_n]\}\in{l_2(\mathbb{N})}$,  then the sequence $c_n=\frac{[f,f_n]}{[f_n,f_n]}$ also belongs to $l_2(\mathbb{N})$.
Hence,  $g_m=\sum_{n=1}^m{c_nf_n}$ is a Cauchy sequence in $\sH$ . The same is true for
the space $(\widehat{\sH}, (\cdot,\cdot)_1)$  since $\{f_n\}$ is a bounded orthogonal basis in $\widehat{\sH}$. 
Therefore,  $\{g_m\}$  is a  Cauchy sequence in $(\mathfrak{D}, (\cdot,\cdot)_{en})$ and  $f$ belongs to  $\mathfrak{D}$.
\end{proof}

Generally, we cannot state that the convergence of \eqref{fr51} is unconditional in $\mathfrak{D}$.
Let us discuss this point in detail. Denote
$$
\mathcal{D}_{un}=\{f\in\sH :  \mbox{the series \eqref{fr51} converges unconditionally in} \ \sH\}.
$$
\begin{theorem}\label{fr75}
Let an $J$-orthogonal Schauder basis  $\cF=\{f_n\}$  be a $J$-frame. Then
$$
\mathcal{D}_{un}\subset\mathcal{M}_{+}\dot{+}\mathcal{M}_{-}\subset\mathfrak{D},
$$
where the inclusions are strict and $f\in\mathcal{D}_{un}$  if  $\{[f,f_n]\}$ belongs to ${l_1}(\mathbb{N})$.
\end{theorem}
\begin{proof}
First of all we note that $\cF=\{f_n\}$ is bounded in $\sH$, i.e., $0<C\leq \|f_n\|^2 \leq D<\infty.$
Indeed,  the $J$-orthonormal Schauder basis $\{f_n/[f_n, f_n]\}$ is bounded in $\sH$ \cite{AZ}.
 This means that $\{f_n\}$ is bounded too (since  \eqref{fr299} holds due to Corollary \ref{fr73}).

 Let $f\in\mathcal{D}_{un}$. Then, simultaneously with \eqref{fr51}, the series
 $\sum_{n\in\mathbb{N}_\pm}c_n{f_n}$ converge to elements $f_\pm$ in the Hilbert space
$\sH$ (see, e.g., \cite[Theorem 3.10]{Heil}).  By the construction $f_\pm\in\mathcal{M}_\pm$. Therefore,
$f=f_++f_-$ belongs to  $\mathcal{M}_{+}\dot{+}\mathcal{M}_{-}$

The sequence $\{f_n\}_{\mathbb{N}_{+}}$ is a basis for the Hilbert spaces $(\mathcal{M}_+, (\cdot,\cdot))$ and
$(\widehat{\mathcal{M}}_+, [\cdot,\cdot])$.  In the last case, $\{f_n\}_{\mathbb{N}_+}$  is  a bounded 
orthogonal basis in $\widehat{\mathcal{M}}_+$. Hence, the relation
\begin{equation}\label{fr74}
f=\sum_{n\in\mathbb{N}_+}c_n{f_n},  \qquad \{c_n\}\in{l_2(\mathbb{N_+})}
\end{equation}
 establishes a one-to-one correspondence between Hilbert spaces  $\widehat{\mathcal{M}}_+$ and  $l_2(\mathbb{N_+})$.

Suppose that $\mathcal{M}_+\subset\mathcal{D}_{un}$. Then the basis $\{f_n\}_{\mathbb{N}_{+}}$ of $(\mathcal{M}_+, (\cdot,\cdot))$ is
unconditional and bounded. Hence, $\{f_n\}_{\mathbb{N}_{+}}$  is a Riesz basis
of $(\mathcal{M}_+, (\cdot,\cdot))$ \cite[Theorem 7.13]{Heil}  and the formula \eqref{fr74} describes all vectors of $\mathcal{M}_+$.
The later means that $\mathcal{M}_+=\widehat{\mathcal{M}}_+$  that is impossible. Therefore,  $\mathcal{M}_+$ does not belong to $\mathcal{D}_{un}$.
The same statement holds for $\mathcal{M}_-$.  The strict inclusion $\mathcal{D}_{un}\subset\mathcal{M}_{+}\dot{+}\mathcal{M}_{-}$
is proved.

Let $f\in\sH$ be such that $\{[f,f_n]\}\in{l_1}(\mathbb{N})$. Then, the sequence $\{{[f,f_n]}/{[f_n,f_n]}\}$ also belongs to
${l_1}(\mathbb{N})$ (since \eqref{fr299} holds). This means that the series \eqref{fr51} converges unconditionally \cite[Lemma 3.5]{Heil}.
Therefore, $f\in\mathcal{D}_{un}$.
\end{proof}

{\bf III.}   Examples of $J$-orthogonal sequences which are $J$-frames can be easily constructed
with the use of Proposition \ref{fr91}.  Indeed, let $\{g_n\}=\{g_n^+\}\cup\{g_n^-\}$  where
$\{g_n^\pm\}$ are orthogonal bounded bases of $\sH_{\pm}$. If
 $g_n\in\mathcal{D}(\cosh{Q/2})$ and the set
 $\{\cosh{Q/2} \ g_n\}$  is complete in $\sH$,  then we obtain a $J$-frame  $\{f_n=e^{-Q/2}g_n\}$, 
 which is a $J$-orthogonal sequence  since
 $$
 [f_n, f_m]=(Je^{-Q/2}g_n,  e^{-Q/2}g_m)=(e^{Q/2}Jg_n,  e^{-Q/2}g_m)=\mbox{sgn}[g_n,g_n]\|g_n\|^2\delta_{nm}.
 $$

 If $Q$ is a bounded operator in $\sH$, then $\{f_n\}$ turns out to be a $J$-orthogonal Riesz basis of $\sH$.
 The case of Schauder basis is characterized by the following conditions (\cite[Theorem 5.12]{Heil}): $Q$ is unbounded
 and  $sup_N\|\sum_{n=1}^N(f, e^{Q/2}g_n)e^{-Q/2}g_n\|<\infty$ for all $f\in\sH$. If the last condition does not hold,
 then the $J$-frame $\{f_n\}$ is an exact sequence of $\sH$.

Let $\sH=L_2(-a, a)$ and $Jf=f(-x)$. Then the subspaces $\sH_{\pm}$ of the fundamental decomposition \eqref{fr7}
coincide with the subspaces of even $L_2^{even}(-a, a)$  and odd $L_2^{odd}(-a, a)$  functions. Let $Q$ be
the multiplication operator $Qf=xf(x)$  in $\sH=L_2(-a, a)$. Obviously, $Q$  anticommutes with $J$ and
the operator $e^{-Q/2}$ acts as the multiplication by $e^{-x/2}$.  Let  $\{g_n\}=\{g_n^+\}\cup\{g_n^-\}$, where
$g_n^\pm$ are orthogonal bounded bases of the subspaces $L_2^{even}(-a, a)$ and $L_2^{odd}(-a, a)$. Then
the $J$-orthogonal sequence $f_n=e^{-x/2}g_n$ is a $J$-frame in the Krein space $(L_2(-a,a), [\cdot, \cdot])$
with the indefinite inner product $[f,g]=\int_{-a}^a{f(-x)}\overline{g(x)}dx$.

{\bf Acknowledgments.}
The authors would like to express our gratitude to Prof. F. Mart\'{i}nez Per\'{i}a for his inspiring lecture on $J$-frames delivered
during STA-2017 conference and forwarded articles which were helpful in the above investigations.

\bibliographystyle{amsplain}

\end{document}